\documentclass[12pt]{amsart}

\usepackage[utf8]{inputenc} 

\usepackage{geometry}

\usepackage{graphicx} 

\usepackage{array} 

\usepackage{setspace}

\usepackage{amsfonts}
\usepackage{amssymb}
\usepackage{amsmath}
\usepackage{amstext}
\usepackage{amscd}
\usepackage{amsthm}
\usepackage{hyperref}
\usepackage[capitalize,nameinlink]{cleveref}
\usepackage{enumerate}
\usepackage{graphicx}

\makeatletter
\@namedef{subjclassname@2010}{%
  \textup{2010} Mathematics Subject Classification}
\makeatother

\newcommand{\e}[0]{\varepsilon}

\newcommand{\ra}[0]{\rightarrow}

\newcommand{\R}[0]{\mathbb{R}}
\newcommand{\Q}[0]{\mathbb{Q}}
\newcommand{\N}[0]{\mathbb{N}}

\newcommand{\Z}[0]{\mathbb{Z}}

\newcommand{\bs}[0]{\backslash}

\newcommand{\set}[1]{\left\{ #1\right\}}

\newcommand{\BT}[0]{\mathbb{T}}

\newcommand{\CC}[0]{\mathcal{C}}

\newcommand{\eKG}[0]{\e\mbox{-Kronecker}}

\newcommand{\ag}[1]{\langle#1\rangle}

\newcommand{\ali}[1]{\begin{align*}#1\end{align*}}

\renewcommand{\Bar}[1]{\overline{#1}}
\renewcommand{\Hat}[1]{\widehat{#1}}
\newcommand{\Ga}[0]{\Gamma}
\newcommand{\ga}[0]{\gamma}

\newcommand{\range}[0]{\operatorname{Range}}
\newcommand{\Proj}[0]{\operatorname{Proj}}
\newcommand{\lcm}[0]{\operatorname{lcm}}

\theoremstyle{theorem}
\newtheorem{thm}{Theorem}[section]
\newtheorem{lem}[thm]{Lemma}
\newtheorem{prop}[thm]{Proposition}
\newtheorem{cor}[thm]{Corollary}

\theoremstyle{definition}
\newtheorem*{defn}{Definition}

\newtheorem*{exmp}{Example}

\theoremstyle{remark}
\newtheorem*{rem}{Remark}

\begin{document}
\baselineskip=17pt

\title[Existence of Large Independent-Like Sets]{Existence of Large Independent-Like Sets}

\author[Robert (Xu) Yang]{Robert (Xu) Yang}
\address{Department of Pure Mathematics University of Waterloo Waterloo ON N2L 3G1, CANADA}
\email{yangxu\_robert@hotmail.com}

\date{}

\begin{abstract}
Let $G$ be a compact abelian group and $\Ga$ be its discrete dual group. For $N \in \N$, we define a class of independent-like sets, $N$-PR sets, as a set in $\Ga$ such that every $\Z_N$-valued function defined on the set can be interpolated by a character in $G$.  

These sets are examples of $\e$-Kronecker sets and Sidon sets. In this paper we study various properties of $N$-PR sets. We give a characterization of $N$-PR sets, describe their structures and prove the existence of large $N$-PR sets.
\end{abstract}



\thanks{Research was partially supported by NSERC grant 2016-03719}

\maketitle

\section{Introduction}

Independence is a property prevalent throughout mathematics. In harmonic analysis, independence has been used to produce curious examples. For example, a perfect independent set in $\R$ was first constructed by Von Neumann \cite{Von}. Independent Cantor sets in non-discrete locally compact abelian groups were constructed by Hewitt and Kakutani \cite{HandK}, in part for showing that $M(G)$ is asymmetric \cite{H}. In these examples, by `independent' we mean algebraically independent: a set of (non-trivial) characters $E$ is independent if whenever $\ga_1,...,\ga_N \in E$, $m_i \in \Z$, and $\prod_{i=1}^N \ga_i^{m_i} = 1$, then $\ga_i^{m_i} = 1$ for all $i$.

The Rademacher functions in the dual of the infinite direct product of infinitely many copies of $\Z_2$ is a set of characters which is both algebraically and probabilistically independent. These functions have proven to be very useful in harmonic analysis. In particular, they have the property that every $\pm 1$-valued function defined on the set is evaluation at some $x$ in the group. A similar interpolation property holds for all algebraically independent sets.

A weaker notion than independence is an $\e$-Kronecker set.
 
\begin{defn}
Let $\e > 0$. The set $E$ is said to be (weak) $\e$-Kronecker if for every $\varphi: E \ra \BT$ there exists $x \in G$ such that $|\varphi(\ga) - \ga(x)| < \e$ for all $\ga \in E$ (resp. $|\varphi(\ga) - \ga(x)| \leq \e$ for all $\ga \in E$).
\end{defn}

This notion was inspired in part by the classical approximation theorem of Kronecker, with early work done by by Hewitt and Kakutani \cite{HandK} and Rudin \cite{Rudin2}. The terminology was introduced by Varapolous in \cite{bigV}.

The set of Rademacher functions is clearly an example of a weak $\sqrt{2}$-Kronecker set. Kronecker-like sets have been studied intensively, and are known to have many interesting properties. For example, Hare and Ramsey \cite{HareR} proved that every $\e$-Kronecker set with $\e < 2$ is a Sidon set. Graham and Hare in \cite{2012paper} introduced the weaker notion of pseudo-Rademacher sets, sets of characters where every $\pm1$-valued function is point-wise evaluation, in order to study the problem of the existence of Kronecker-like sets. Galindo and Hernandez in \cite{GaandH} and Graham and Lau in \cite{GandL} both consider interpolation sets of characters of finite order. For other references and further background information we refer the reader to \cite{Harebook}.

In this paper, we generalize this notion to $N$-pseudo-Rademacher sets (or $N$-PR sets for short), sets of characters with the property that every $\Z_N$-valued function on the set is point-evaluation. Of course, a pseudo-Rademacher set is a $2$-PR set and we will see that $N$-PR sets are $\e$-Kronecker sets for suitable $\e = \e(N)$, which tends to $0$ as $N \ra \infty$. We give an algebraic characterization of $N$-PR sets, compare them with $\e$-Kronecker sets, describe their structures and prove existence theorems (\cref{thm2.8}) of large $N$-PR sets. \cref{thm2.8} gives a new proof that any uncountable subset in $\Ga$ contains a large $\e$-Kronecker set.

\section{Characterization of $N$-PR sets}
Throughout this paper $G$ is a compact abelian group and $\Ga$ is its discrete dual group. 

\begin{defn} 
Let $E \subset \Ga$ be a subset and $N \in \N$. We define $E$ to be an  ``$N$-pseudo-Rademacher" set (or $N$-PR set) if for every $\varphi: E \ra \Z_N \subset \BT$, there exists $x \in G$ such that $\varphi(\ga) = \ga(x)$ for all $\ga \in E$. 
\end{defn}

In this section we give an algebraic characterization of $N$-PR sets. We first establish some useful lemmas.

\begin{lem}[Quotient and Subgroup Lemma] \label{lem2.1}
Let $E \subset \Ga$, $n \in \N$, $\ga \in \Ga$ and $\Lambda \subset \Ga$ be a subgroup.

(1) Let $q: \Ga \ra \Ga/\Lambda$ be the quotient map. If $q$ is one-to-one on $E$ and $q(E)$ is $n$-PR, then $E$ is $n$-PR. 

(2) Suppose $E \subset \Lambda$. Then $E$ is $n$-PR as a subset of $\Ga$ if and only if $E$ is $n$-PR as a subset of $\Lambda$.
\end{lem}

\begin{proof}
(1) Suppose $q: \Gamma \ra \Ga / \Lambda$ is one-to-one on $E$ and $q(E)$ is $n$-PR. We will show that $E$ is $n$-PR. Let $\varphi: E \ra \Z_n$ be a function. Because $q$ is one-to-one on $E$, for each $\ga, \beta \in E$, if $\beta \neq \ga$, then $\ga - \beta \notin \Lambda$. Thus, we can define $\varphi': q(E) \ra \Z_n$ via $\varphi'(\ga + \Lambda) = \varphi(\ga)$ for $\ga \in E$. Since $q(E)$ is $n$-PR, there exists $x \in \Lambda^{\perp} = \Hat{\Gamma / \Lambda}$ such that $\varphi'(\ga + \Lambda) = x(\ga + \Lambda)$ for all $\ga \in E$. As $x \in \Lambda^\perp$, $\varphi(\ga) = \ga(x)$ for all $\ga \in E$. This means $E$ is $n$-PR.

(2)  We first suppose $E$ is an $n$-PR subset of $\Ga$. Let $\varphi: E \ra \Z_n$ be a function. There exists $x \in G$ such that $\varphi(\ga) = \ga(x)$ for all $\ga \in E$. Let $x+ \Lambda^\perp \in G/\Lambda^\perp = \Hat{\Lambda}$. Since $E \subset \Lambda$, for all $\ga \in E$ we have
\ali{\varphi(\ga) =  \ga(x) = (x+ \Lambda^\perp)(\ga).}
This means $E$ is $n$-PR as a subset of $\Lambda$. The proof of the converse part of (2) is similar.
\end{proof}

The following proposition is a stronger version of Lemma 3.2 in \cite{KandR}.

\begin{prop}\label{prop2.2}
Let $E \subset \Ga$. The following are equivalent:

(1)  For all $\varphi: E \ra \BT$ with $\varphi(\ga) \in \range(\ga)$ there exists $x \in G$ such that $\varphi(\ga) = \ga(x)$ for $\ga \in E$.
 
(2)  $E$ is independent. 
\end{prop}

\begin{proof}
Assume $E$ is independent and that $\varphi: E \ra \BT$ satisfies $\varphi(\ga) \in \range(\ga)$ for $\ga \in E$. By similar arguments to the proof of \cref{lem2.1}, if we can find $x \in \Hat{\ag{E}}$ such that $\varphi(\ga) = \ga(x)$ for all $\ga \in E$, then there exists $x' \in G$ such that $\varphi(\ga) = \ga(x')$. Thus we may assume $\Ga = \ag{E} = \bigoplus_{\ga \in E} \ag{\ga}$. For each $\ga \in E$, there exists $x_\ga \in \Hat{\ag{\ga}}$ such that $\ga(x_\ga) = \varphi(\ga)$. If $E$ is finite, we let $x = \prod_{\ga \in E} x_\ga$ and $x$ can interpolate $\varphi$ exactly. For the case that $E$ is infinite, since $G$ is compact, we let $x \in G$ be a cluster point of the following set 
\ali{ \set{\prod_{\ga \in F} x_\ga : F \subset E, |F| < \infty},}
and such an $x$ can interpolate $\varphi$ exactly. 

Conversely, if $E$ is not independent, then there exist $\ga_1,...,\ga_k \in E$ and $m_1,...,m_k \in \Z$ such that $\ga_1^{m_1}...\ga_k^{m_k} = 1$ but $\ga_i^{m_i} \neq 1$ for all $1 \leq i \leq k$. Consider the function $\varphi: E \ra \BT$ such that $\varphi(\ga_i) = 1$ for all $i > 1$, $\varphi(\ga_1)^{m_1} \neq 1$ and $\varphi(\ga) \in \range(\ga)$ for all $\ga \in E$. Notice that such $\varphi$ exists, because $\ga_1^{m_1} \neq 1$ means $\range(\ga_1^{m_1}) \neq \set{1}$ and hence there exists $x \in G$ such that $\ga_1(x)^{m_1}  \neq 1$. Let $\varphi(\ga_1) = \ga_1(x)$.
This function $\varphi$ cannot be interpolated by any $x \in G$. 
\end{proof}

\begin{cor}
Let $N \in \N$ and $E \subset \Ga$. If $E$ is independent and $\Z_N \subset \range(\ga)$ for all $\ga \in E$, then $E$ is $N$-PR.
\end{cor}
\begin{proof}
This follows directly from \cref{prop2.2}.
\end{proof}
 
\begin{thm}\label{thm2.3}
The following are equivalent:

(1) $E \subset \Ga$ is $N$-PR.

(2) If $\ga_i \in E$ for $1 \leq i \leq n$ are distinct and $\prod_{i = 1}^n \ga_i^{m_i} = 1$ for some $m_i \in \Z$, then $N$ divides $m_i$ for all $i$.
\end{thm}

\begin{proof}
Suppose (2) fails. Then there exist distinct $\ga_i \in E$ and integers $m_i \in \Z$ with $\prod_{i =1}^n \ga_i^{m_i} = 1$, while $N$ does not divide $m_1$. Let $f: E \ra \Z_N$ be given by $f(\ga_1) = e^{2\pi i /N}$ and $f(\ga) = 1$ for all $\ga \neq \ga_1$. For any $x \in G$, $1 = \prod_{i =1}^n \ga_i(x)^{m_i}$, while $\prod_{i=1}^n f(\ga_i)^{m_i} = f(\ga_1)^{m_1} \neq 1$. Thus, this function $f$ cannot be interpolated by any $x \in G$ and therefore (1) fails.

Conversely, suppose (2) holds. Let $E_N = \ag{\ga^N : \ga \in E}$, the subgroup generated by $\set{\ga^N: \ga \in E}$, and $\pi: \Ga \ra \Ga / E_N$ be the quotient map. Elements in $E_N$ have the form $\ga_1^{k_1N}...\ga_m^{k_mN}$ for $\ga_1,...\ga_m \in E$ and $k_1,...,k_m \in \Z$. We claim that $\pi(E) \subset \pi(\Ga)$ is independent. Indeed, suppose $\ga_i \in E$ with distinct $\pi(\ga_i)$ and $m_i \in \Z$ such that $\prod_{i =1}^n \ga_i^{m_i} \in \ker(\pi)$. Then there exists $k_j \in \Z$, $1 \leq j \leq s$ and $s \geq n$  such that $\prod_{i =1}^n \ga_i^{m_i} = \prod_{j =1}^s \ga_j^{k_j N}$ with distinct $\ga_j \in E$. Hence, $\prod_{i = 1}^n \ga_i^{m_i - k_iN} \prod_{j = n+1}^s \ga_j^{-k_jN} = 1 \in \Ga$. By (2), $N \ | \ m_i - k_iN$ for all $1 \leq i \leq n$ and hence $N \ | \ m_i$ for all $1 \leq i \leq n$. This means each $\ga_i^{m_i} \in \ker(\pi)$ and therefore the set $\pi(E)$ is independent. 

Moreover, we claim that $\pi$ is one-to-one on $E$. Suppose, otherwise, that there are $\ga_1 \neq \ga_2 \in E$ and $\ga_1 \ga_2^{-1} \in \ker(\pi)$. Then there exists $k_j \in \Z$, $1 \leq j \leq s$ and $s \geq 2$  such that $\ga_1 \ga_2^{-1} = \prod_{j =1}^s \ga_j^{k_j N}$ with distinct $\ga_j \in E$. We have $\ga_1^{1 - k_1N} \ga_2^{-1-k_2N} \prod_{j = 3}^s \ga_j^{-k_j N} = 1$. Again, (2) gives $N \ | \ 1 - k_1N$ and $N \ | \ - 1 - k_2N$, which are not possible.

Similar arguments show that elements of $\pi(E)$ have order $N$. Thus $\pi(E)$ is $N$-PR from \cref{prop2.2}. Because $\pi$ is also one-to-one on $E$, \cref{lem2.1} gives that $E$ is $N$-PR, proving (1).
\end{proof}

\begin{cor}\label{cor2.4}
Let $a,b \in \N$ be co-prime. The subset $E \subset \Ga$ is $(ab)$-PR if and only if $E$ is both $a$-PR and $b$-PR.
\end{cor}

\begin{proof}
Since $\Z_a, \Z_b \subset \Z_{ab}$, if $E$ is $(ab)$-PR, $E$ is both $a$-PR and $b$-PR. To see the converse, we assume $E$ is both $a$-PR and $b$-PR. Consider $\ga_1,...,\ga_n \in E$ and $m_1,...,m_n \in \Z$ such that $\prod_{i=1}^n \ga_i^{m_i} = 1$. Since $E$ is both $a$-PR and $b$-PR, $a | m_i$ and $b | m_i$ for all $1 \leq i \leq n$. Since $a$ and $b$ are co-prime, $ab | m_i$ for all $1\leq i \leq n$. Hence, by \cref{thm2.3}, $E$ is $(ab)$-PR.
\end{proof}

\begin{cor}\label{cor2.5}
Suppose $E \subset \Ga$ is $N$-PR and $\ag{E} \cap \ag{\ga} = \set{1}$. Then $\ga E$ is also $N$-PR.
\end{cor}

\begin{proof}
Let $\ga_1,...\ga_n \in E$ and $m_1,...,m_i \in \Z$ be such that $\prod_{i=1}^n (\ga\ga_i)^{m_i} = 1$.  Since $\ag{E} \cap \ag{\ga} = \set{1}$, we have $\prod_{i=1}^n \ga_i^{m_i} = 1$ and therefore $N | m_i$ for all $1 \leq i \leq n$. This shows $\ga E$ is $N$-PR by \cref{thm2.3}.
\end{proof}

Next we compare $N$-PR sets with $\e$-Kronecker sets. We first note that $N$-PR sets are weak $\e$-Kronecker sets for appropriate $\e$ depending on $N$.
 
\begin{prop}
$N$-PR sets are weak $\eKG$ for $\e = |1 - e^{\frac{\pi i}{N}}|$.
\end{prop}
\begin{proof}
Assume $E \subset \Ga$ is $N$-PR. Let $\varphi: E \ra \BT$. Define $\varphi_N: E \ra \Z_N$ by $\varphi_N(\ga) = t$, where $t \in \Z_N$ satisfies $|t - \varphi(\ga)| \leq |1 - e^{\frac{\pi i}{N}}|$. Let $x \in G$ be such that $\ga(x) = \varphi_N(\ga)$ for $\ga \in E$. We have $|\ga(x) - \varphi(\ga)| \leq |1 - e^{\frac{\pi i}{N}}|$.
\end{proof}

Of course, not every $\e$-Kronecker set in a torsion group is $N$-PR. Here is an example.

\begin{exmp}
For every $\e > 0$ and $N \in \N$, there exists an infinite $\eKG$ set $E$ with every $\ga \in E$ having finite order a multiple of $N$, but $E$ is not $N$-PR.  

Let $N \in \N$ and $\e > 0$. Let $(p_i)_{i = 1}^\infty$ be an increasing sequence of primes coprime to $N$. Let $\Ga = \Z_N \oplus \bigoplus_{i \geq 1} \Z_{p_i}$. Let 
\ali{S := \set{(1,0,0,...), (1, 1, 0,0,...), (1,1,1,0,0,...),...}.}
Each element in $S$ has order a multiple of $N$. It is not hard to see that if we exclude finitely many elements of small orders, we have a co-finite $\eKG$ set $E \subset S$. From \cref{thm2.3} no subsets of $S$ other than singletons are $N$-PR and therefore $E$ is not $N$-PR.  
\end{exmp}

\begin{rem}
Recall that the Bohr compactification of $\Ga$, denoted by $\Bar{\Ga}$, is defined as the dual group of $G_d$, where $G_d$ is the group $G$ equipped with the discrete topology.

Let $\Bar{E}$ be the closure of $E$ in $\Bar{\Ga}$. In  \cite{Harebook}, Ex. 1.4.3, it is shown that there is an $\eKG$ set $E$ and an integer $M \geq 1$ such that $\Bar{(E \cup E^{-1})^M}$ has Haar measure 1. That is not the case for $N$-PR sets. Indeed, if $E$ is any $N$-PR set with $N \geq 3$, then the Haar measure of $\Bar{(E \cup E^{-1})^M}$ is $0$ for all $M \geq 1$. The proof is very similar to the proof of Theorem 3.1 in \cite{CCHare2} and therefore is omitted. 
\end{rem}

\section{Structure of $N$-PR sets}

In this section, we investigate the structure of $N$-PR sets. We rely heavily on the following structure theorem for general abelian groups.

Notation: Let $p$ be a prime number. The group $\CC(p^\infty)$ is the group of all $p^n$-th roots of unity for $n \geq 1$.
 
\begin{thm}\label{thm3.1}\cite{structure}
Every abelian group $\Ga$ is isomorphic to a subgroup of 
\ali{\bigoplus_{\alpha} \Q_\alpha \oplus \bigoplus_{\beta} \CC(p_\beta^\infty),}
where  $\Q_\alpha$ are copies of $\Q$ and $p_\beta$ are prime numbers.
\end{thm}
 
Notation: (1) We let $\Ga_0$ be the torsion subgroup of $\Ga$ and $\pi_0: \Ga \ra \Ga / \Ga_0$ be the quotient map. In the notation of \cref{thm3.1}, $\pi_0: \bigoplus_{\alpha} \Q_\alpha \oplus \bigoplus_{\beta} \CC(p_\beta^\infty) \ra \bigoplus_{\alpha} \Q_\alpha$ is the quotient map.

(2) Let $p \in \N$ be a prime number and $n \in \N$. We let $\Ga_{p^n}$ be the subgroup of $\Ga_0$ containing elements whose orders are not a multiple of $p^n$, or equivalently, whose orders are not divisible by $p^n$. Notice $\Ga_{p^n}$ is indeed a subgroup because if $\ga_1$ and $\ga_2$ have orders $m_1$ and $m_2$, neither a multiple of $p^n$, then $\lcm(m_1, m_2)$, the order of $\ga_1\ga_2$, is not a multiple of $p^n$. Let $\pi_{p^n}: \Ga \ra \Ga / \Ga_{p^n}$ be the quotient map.

(3) For $\ga \in \bigoplus_{\alpha \in A} \Q_\alpha \oplus \bigoplus_{\beta \in B} \CC(p_\beta^\infty)$ and $i \in A \cup B$, we let $\Proj_i(\ga)$ be the $i$-th coordinate of $\ga$. In particular, if $\Ga = \bigoplus_{i} \CC(p_i^\infty)$ with $p_i$ distinct, the map $\pi_{p_i}$ can be viewed as $\pi_{p_i} = \Proj_i$.

\begin{prop}\label{prop3.2}
 Let $E \subset \Ga$, $n \in \N$ and $p$ be a prime number. The following are equivalent:
 
(1) $E$ is $p^n$-PR.

(2) $\pi_{p^{k}}(E)$ is $p^{n+1-k}$-PR and $\pi_{p^k}$ is one-to-one on $E$ for all $1 \leq k \leq n$.

(3) $\pi_{p^{k}}(E)$ is $p^{n+1-k}$-PR and $\pi_{p^k}$ is one-to-one on $E$ for some $1 \leq k \leq n$. 
\end{prop}

\begin{proof}
We first show (1) implies (2). Fix $1 \leq k \leq n$ and we suppose $E \subset \Ga$ is $p^n$-PR.  We first claim that $\pi_{p^k}$ is one-to-one on $E$. This is because if $\ga_1, \ga_2 \in E$ have $\ga_1\ga_2^{-1} \in \Ga_{p^k}$, then $\ga_1\ga_2^{-1}$ has finite order that is not divisible by $p^k$. But from \cref{thm2.3}, this implies $E$ is not even $p^k$-PR and therefore contradicts that $E$ is $p^n$-PR. Next, we show that $\pi_{p^k}(E)$ is $p^{n+1-k}$-PR. We let $\ga_1,...,\ga_s \in E$ and $\beta_i := \pi_{p^k}(\ga_i)$ for $1 \leq i \leq s$. Suppose for some $m_1,...,m_s \in \Z$ we have $\prod_{i=1}^s \beta_i^{m_i} = 1$. This means $\prod_{i=1}^s \ga_i^{m_i} \in \Ga_{p^k}$ and therefore there exists $l \in \N$, which is not divisible by $p^k$, such that $\prod_{i=1}^s \ga_i^{lm_i} = 1$. By \cref{thm2.3}, $p^n | lm_i$ and hence $p^{n+1-k} | m_i$. \cref{thm2.3} implies $\pi_{p^k}(E)$ is $p^{n+1-k}$-PR.

 Since (2) implies (3) is obvious, it remains to show (3)  implies (1). We assume $(3)$ holds for some $1 \leq k \leq n$. Let $\ga_1,...,\ga_s \in E$ be distinct and $m_1,...,m_s \in \Z$ such that $\prod_{i=1}^s \ga_i^{m_i} = 1$. Then $\prod_{i=1}^s \pi_{p^k}(\ga_i)^{m_i} = 1$ and the injectivity of $\pi_{p^k}$ implies $\pi_{p^k}(\ga_i)$ are distinct. Since $\pi_{p^k}(E)$ is $p^{n+1-k}$-PR, we have $p^{n+1-k} | m_i$ for all $1 \leq i \leq s$. 

If $n+1-k \geq k$, we note that order of the product $\prod_{i=1}^s \ga_i^{m_i/p^{k-1}}$ divides $p^{k-1}$ and therefore $\prod_{i=1}^k \ga_i^{m_i/p^{k-1}} \in \Ga_{p^k}$. This means $\prod_{i=1}^k \pi_{p^k}(\ga_i)^{m_i/p^{k-1}} = 1 \in \Ga/\Ga_{p^k}$. Hence, by \cref{thm2.3}, we have $p^{n+1-k}$ divides $m_i/p^{k-1}$ and this gives $p^n$ divides $m_i$ for all $1 \leq i \leq s$. \cref{thm2.3} gives (1).

Otherwise, we have $n+1-k < k$. Notice that the order of the product $\prod_{i=1}^s \ga_i^{m_i/p^{n+1-k}}$ divides $p^{n+1-k}$. Since $n+1-k < k$, $\prod_{i=1}^k \ga_i^{m_i/p^{n+1-k}} \in \Ga_{p^k}$. As above, we have $p^{n+1-k}$ divides $m_i/p^{n+1-k}$, which means $p^{2(n+1-k)} | m_i$. We continue doing this until we reach $r(n+1-k) \geq k$ for some $r \in \N$. The previous case gives (1).
\end{proof}

\begin{cor}\label{cor3.3}
Suppose $\Ga = \bigoplus \CC(p^\infty)$. Then $p^n$-PR subsets in $\bigoplus \CC(p^\infty)$ are in correspondence with $p$-PR subsets in $ \bigoplus \CC(p^\infty)$ in the following manner: Given a set $E \subset \bigoplus \CC(p^\infty)$, we let $E_{p^n} := \set{\ga^{p^{n-1}} : \ga \in E}$ and $E_{1/p^n} := \set{\xi_\ga : \ga \in E}$, where $\xi_\ga^{p^{n-1}} = \ga$. If $E$ is $p^n$-PR then $E_{p^n}$ is $p$-PR and $|E_{p^n}| = |E|$. If $E$ is $p$-PR then $E_{1/p^n}$ is $p^n$-PR and $|E| = |E_{1/p^n}|$.
\end{cor}
\begin{proof}
Since on $\bigoplus \CC(p^\infty)$, the map $\pi_{p^n}$ can be identified as $\pi_{p^n}(\ga) = \ga^{p^{n-1}}$, this follows from \cref{prop3.2}.
\end{proof}

\begin{cor}\label{cor3.4}
Let $N = p_1^{n_1}...p_k^{n_k}$ for distinct primes $p_i$ and $n_i \in \N$. Then $E$ is $N$-PR if and only if $\pi_{p_i^{n_i}}(E)$ is $p_i$-PR and $\pi_{p_i^{n_i}}$ is one-to-one on $E$ for all $1 \leq i \leq k$.
\end{cor}
\begin{proof}
This follows from \cref{cor2.4} and \cref{prop3.2}.
\end{proof}

We thus have the following result about the structure of $N$-PR sets in the torsion subgroup.

\begin{prop}\label{prop3.5}
Let $E \subset \Ga_0$ be an $N$-PR set and $N = p_1^{n_1}...p_k^{n_k}$ where the $p_i$ are distinct prime numbers. There exist $p_i^{n_i}$-PR sets $E_i \subset \bigoplus \CC(p_i^\infty)$ for $1 \leq i \leq k$, and bijections $f_2: E_1 \ra E_2$, ..., $f_{k}: E_1 \ra E_k$ such that we can represent $E$ as
\ali{E = \set{\ga \oplus f_2(\ga) \oplus ... \oplus f_{k}(\ga) \oplus \beta_\ga : \ga \in E_1}}
for some $\beta_\ga \in \bigoplus_{p \neq p_i \ \forall 1 \leq i \leq k} \CC(p^\infty)$.
\end{prop} 
\begin{proof}
For $1 \leq i \leq k$, we let $E_i := \pi_{p_i}(E)$. Since $E$ is $N$-PR, from \cref{prop3.2}, each $E_i$ is $p_i^{n_i}$-PR and the maps $\pi_{p_i}: E \ra E_i$ are injective. For $2 \leq i \leq k$, we define $f_i$ on $E_1$ as $f_i(\pi_{p_1}(\ga)) := \pi_{p_i}(\ga)$ for $\ga \in E$. The injectivity of $\pi_{p_1}$ on $E$ ensures the maps are well-defined. Moreover, the injectivity of $\pi_{p_i}$ implies $f_i$ is injective. Each $f_i$ is clearly surjective by its construction and therefore a bijection. Since each $\ga \in E$ can be represented as 
\ali{\ga = \oplus_{i=1}^k \pi_{p_i}(\ga) \oplus \beta_\ga}
for some $\beta_\ga \in  \bigoplus_{p \neq p_i \ \forall 1 \leq i \leq k} \CC(p^\infty)$,  \cref{prop3.5} follows.
\end{proof}

Next, we will discuss the relation between $p$-PR sets and independent sets. 
 
We first note that not every $p$-PR set is an independent set. The following example shows a $p$-PR set whose only independent subsets are singletons.

\begin{exmp}
Consider $E = \set{\ga_n : n \geq 2} \subseteq \bigoplus_{i \geq 1} \CC(p^\infty)$, where $\Proj_1(\ga_n) = 1/p^2$, $\Proj_n(\ga_n) = 1/p$ and $\Proj_k(\ga_n) = 0$ for all $k \neq 1, n$. Here we use additive group operation. $E$ is $p$-PR by \cref{thm2.3}, but $E$ does not contain any independent subsets other than singletons because for all $i \neq j$, $p\ga_i = p\ga_j \neq 1$.
\end{exmp}

Similar to independent sets, we have the following result about the maximum size of a $p$-PR set inside a product group.

\begin{prop}\label{prop4.2}
Suppose $E \subseteq \bigoplus_{i \in B_1} \Q \oplus \bigoplus_{i \in B_2} \CC(p^\infty)$ is $p$-PR. Then $|E| \leq |B_1|+|B_2|$.
\end{prop} 
\begin{proof}
We identify each element in $\CC(p^\infty)$ in the form of $a/p^n$ for some $a, n \in \N$ with additive group operation. Hence, we can identify $ \bigoplus_{i \in B_1} \Q \oplus \bigoplus_{i \in B_2} \CC(p^\infty)$ as a subset of the real vector space $\R^{|B_1|+|B_2|}$. 

We claim that if $E \subseteq  \bigoplus_{i \in B_1} \Q \oplus \bigoplus_{i \in B_2} \CC(p^\infty)$ is $p$-PR, $E$ is linearly independent in $\R^{|B_1|+|B_2|}$. Indeed, suppose that $E$ is not linearly independent. Then there exist $\set{\ga_i \ | \ 1 \leq i \leq k} \subseteq E$ and $a_1,...,a_k \in \R$ such that $a_1\ga_1 + ... +a_k \ga_k = 0$, while not all $a_i$'s are zero and each $\ga_i \in\R^{|B_1|+|B_2|}$ is identified as above. Since the entries of each $\ga_i$ are in $\Q$, we may assume $a_i \in \Q$ for all $1 \leq i \leq k$. Furthermore, we may assume $a_i \in \Z$. Notice that if $p | a_i$ for all $i$, we may replace $a_i$ by $a_i/p$. Hence, we may find a choice of such $a_i$'s, not all  divisible by $p$. This contradicts that $E$ is $p$-PR by \cref{thm2.3}. 

Thus, $|E| \leq  |B_1|+|B_2|$.
\end{proof}

\section{Existence of $N$-PR Sets}

In this section, we show some existence results about $N$-PR sets and that large $N$-PR sets are plentiful. We first prove a lemma.

\begin{lem}\label{lem4.3}
Assume $E \subseteq \bigoplus_{\beta \in B} \Ga_\beta$ is uncountable.

(1) Let $p$ be a prime number. If $\Ga_\beta = \CC(p^\infty)$ for all $\beta \in B$, then $E$ contains a $p$-PR subset of the same cardinality.

(2) If $\Ga_\beta = \Q$ for all $\beta \in B$, then $E$ contains an independent subset of the same cardinality.
\end{lem}
\begin{proof}
Without loss of generality, we further assume for each $\beta \in B$ there exists $\ga \in E$ such that $\Proj_\beta(\ga)$ is non-trivial.
(1) We first prove (1) in the special case that every $\ga \in E$ has order $p$. We consider the collection $\CC$ of subsets of $E$ defined as $A \in \CC$ if for all finite subsets $F \subseteq A$ there exists an arrangement $F = \set{\ga_1,...,\ga_n}$ such that for each $1 \leq k \leq n$ there is some $\beta \in B$ with $\Proj_\beta(\ga_k)$ non-trivial, but $\Proj_\beta(\ga_j)$ trivial for all $1 \leq j \leq k$. We partially order $\CC$ by inclusion and Zorn's Lemma gives a maximal $S \in \CC$.

We claim $|S| = |E|$. Indeed, if $|S| < |E|$, we let $B_1 \subset B$ be given by $\beta \in B_1$ if there exists $\ga \in S$ such that $\Proj_\beta(\ga)$ is non-trivial. We thus have $|B_1| < |B|$. Let $\beta_0 \in B \bs B_1$ and $\ga_0 \in E$ be such that $\Proj_{\beta_0}(\ga_0)$ is non-trivial. Since $\beta_0 \in B \bs B_1$, $\ga_0 \notin S$ and we form the set $S_1 := S \cup \set{\ga_0}$. It is easy to see $S_1 \in \CC$ and this contradicts the maximality of $S$.

Moreover, the construction of $S$, the assumption that every element has order $p$ and \cref{thm2.3} imply $S$ is $p$-PR, which finishes the proof for the special case.

For the general case, we let $E_k$ be defined as the subset in $E$ containing elements of order $p^k$. Then $E = \cup_{k \geq 0} E_k$ and hence there exists a positive integer $K \in \N^+$ such that $|E_K| = |E|$. If $K = 1$, the special case finishes the proof, and therefore we suppose $K > 1$.

We have two cases. The first case is that there exists $1 \leq n_0 < K$ such that $|\pi_{p^{n_0}}(E_K)| = |E_K|$, while $ |\pi_{p^{n_0+1}}(E_K)| < |E_K|$. We let $B_1 \subset B$ be given by $\beta \in B_1$ if there exists $\ga \in E_K$ such that $\Proj_\beta(\ga)$ has order greater or equal to $p^{n_0+1}$. Since $|\pi_{p^{n_0+1}}(E_K)| < |E_K|$, we have $|B_1| < |B| = |E_K|$. For a subset $C \subset B$, we define the projection $\Proj_C: \bigoplus_{\beta \in B} \CC(p^\infty) \ra \bigoplus_{\beta \in C} \CC(p^\infty)$. We have $|\Proj_{B \bs  B_1}(\pi_{p^{n_0}}(E_K))| = |E_K|$. Furthermore, each $\alpha \in \Proj_{B \bs  B_1}(\pi_{p^{n_0}}(E_K))$ has the property that for all $\beta \in B \bs  B_1$, $\Proj_\beta(\alpha)$  has order $p$ or $1$. Thus, the special case gives a $p$-PR set $F \subset  \Proj_{B \bs  B_1}(\pi_{p^{n_0}}(E_K))$ such that $|F| = |\Proj_{B \bs  B_1}(\pi_{p^{n_0}}(E_K))| = |E_K| = |E|$ and the quotient lemma finishes the proof for this case.

The other case is that $|\pi_{p^K}(E_K)| = |E_K|$. Since $\pi_{p^K}(E_K)$ satisfies the special case, the quotient lemma again finishes the proof.

(2) The proof of (2) is similar to the first part of the argument of (1), but much simpler.
\end{proof}

\begin{thm}\label{thm2.8}
Let $E \subset \Ga$ be uncountable. Then there exists a prime number $p$ such that $E$ contains a $p$-PR set of the same cardinality.
\end{thm}
\begin{proof}
Embed $E \subset \bigoplus_{i \in B_0} \Q \oplus \oplus_{j=1}^\infty \bigoplus_{i \in B_j} \CC(p_j^\infty)$, where $(B_j)_{j=0}^\infty$ are index sets and $p_j$ are distinct primes. We assume that for each index $i \in \bigcup_{j=0}^\infty B_j$, there exists $\ga \in E$ such that $\Proj_i(\ga)$ is non-trivial. Since $E$ is uncountable and the groups $\Q$ and $\CC(p_j^\infty)$ are countable, there exists $K \in \N$ such that $|B_K| = |E|$.

 If $K = 0$, from \cref{lem4.3} (2) we may extract an independent set $F \subset \pi_0(E)$ with $|F| = |E|$. If we choose $E' \subset E$ such that $\pi_0$ is one-to-one on $E'$ and $\pi_0(E') = F$, then $E'$ is $N$-PR for all $N \geq 2$ by \cref{lem2.1} and \cref{prop2.2}.
 
Similarly, if $K \geq 1$, by \cref{lem4.3} (1) we extract a $p_K$-PR subset $F \subset \pi_{p_K}(E)$ with $|F| = |E|$ and therefore obtain a $p_K$-PR subset in $E$ of the same cardinality.
\end{proof}

\begin{rem}
Since any $p$-PR set is a Sidon set, \cref{thm2.8} implies any uncountable subset in $\Ga$ contains a large Sidon set. $I_0$ sets are  more special interpolation sets. A set $E \subset \Ga$ is $I_0$ if every bounded function defined on $E$ can be interpolated as a Fourier transform of a discrete measure. If $p \geq 3$, any $p$-PR set is an $I_0$ set, and therefore in that case any uncountable subset in $\Ga$ contains a large $I_0$ set.
\end{rem}

\begin{thm}\label{cor2.7}
Let $E \subset \Ga$ be uncountable, $p$ be a prime number and $n \in \N$. Then $E$ contains a $p^n$-PR subset of the same cardinality if and only if $|\pi_{p^n}(E)| = |E|$.
\end{thm}
\begin{proof}
If $E$ contains a $p^n$-PR subset $E_1$ of the same cardinality, by \cref{prop3.2} $|\pi_{p^n}(E)| \geq  |\pi_{p^n}(E_1)| = |E_1| = |E|$.

To see the converse, we define the projection 
\ali{\pi_1: \bigoplus_{i \in B_1} \Q \oplus \bigoplus_{i \in B_2} \CC(p^\infty) \ra \bigoplus_{i \in B_2} \CC(p^\infty).}
Assume $|\pi_{p^n}(E)| = |E|$. We claim that either $|\pi_0(\pi_{p^n}(E))| = |E|$ or $|\pi_1(\pi_{p^n}(E))| = |E|$. 

If $|\pi_0(\pi_{p^n}(E))| < |E|$ and $|\pi_1(\pi_{p^n}(E))| < |E|$, then because $E$ is uncountable,
\ali{|\pi_{p^n}(E)| \leq |\pi_0(\pi_{p^n}(E))||\pi_1(\pi_{p^n}(E))| < |E|,}
which is a contradiction. If $|\pi_0(\pi_{p^n}(E))| = |E|$, we appeal to (2) in \cref{lem4.3} to get a $p^n$-PR set in $\pi_0(\pi_{p^n}(E))$ and \cref{lem2.1} finishes the proof. If $|\pi_1(\pi_{p^n}(E))| = |E|$, then (1) in \cref{lem4.3} similarly finishes the proof.
\end{proof}

\begin{exmp}
If $E$ is countable, \cref{thm2.8} may fail. Let $E = \CC(p^\infty)$. $E$ is countably infinite, but by \cref{prop4.2}, $E$ only contains $p$-PR sets that are singletons.
\end{exmp}

\begin{prop}\label{prop3.6}
Suppose $E = \Ga \subset \bigoplus_{i \in B_1} \Q \oplus \bigoplus_{i \in B_2} \CC(p^\infty)$ is a countably infinite group. Then $E$ contains an infinite $p$-PR set if and only if $E$ contains an infinite independent set.
\end{prop}
\begin{proof}
Assume that all independent sets in $E$ are finite. Zorn's Lemma gives a maximal independent set $S \subset E$, where the partial order is inclusion. Our assumption gives that $S$ is finite and the maximality implies $\ag{S} = E$. Since $S$ is independent, we embed $S$ into $\bigoplus_{\ga \in S} \Ga_\ga$, where $\Ga_\ga = \Q$ if $\ga$ has infinite order and $\Ga_\ga = \CC(p^\infty)$ if $\ga$ has order some power of $p$. Since $S$ is maximal, we may extend  the embedding to $\ag{S} = E$. From \cref{prop4.2}, this implies all $p$-PR sets in $E$ have cardinality at most $|S|$.
The converse is trivial.
\end{proof}

\begin{prop}
Suppose $\Ga$ is an uncountable infinite group and $N=p_1^{m_1}...p_k^{m_k}$ is an integer with prime numbers $p_i$, $1 \leq i \leq k$. Then $\Ga$ contains an $N$-PR set $E$ with $|E| = |\Ga|$ if and only if $|\pi_{p_i^{m_i}}(\Ga)| = |\Ga|$ for all $1 \leq i \leq k$.
\end{prop}
\begin{proof}
If $\Ga$ contains an $N$-PR set $E$ with $|E| = |\Ga|$, then by \cref{cor3.4}  $|\pi_{p_i^{m_i}}(\Ga)| = |\Ga|$ for all $1 \leq i \leq k$. 

To see the converse, first of all, if $|\pi_0(\Ga)| = |\Ga|$, since $\Ga$ is uncountable, the conclusion follows by \cref{lem4.3} and \cref{lem2.1}. Thus we may assume $|\pi_0(\Ga)| < |\Ga|$. Hence $|\Ga_0| = |\Ga|$ and we may further assume $\Ga$ is a torsion group. If $|\pi_{p_i^{m_i}}(\Ga)| = |\Ga|$ for all $1 \leq i \leq k$, then by \cref{lem4.3} (1) we let $S_i$ be a subset in $\pi_{p_i^{m_i}}(\Ga)$ such that $S_i$ is $p_i$-PR with $|S_i| = |\Ga|$, $1 \leq i \leq k$. For $1 \leq i \leq k$ we let $J_i \subset \Ga$ be such that $\pi_{p_i^{m_i}}$ is one-to-one on $J_i$ and $\pi_{p_i^{m_i}}(J_i) = S_i$ and moreover, we may assume the order of each $\ga \in J_i$ is a power of $p_i$. Since $|J_i| = |\Ga|$ for all $1 \leq i \leq k$, we let, for $2 \leq i \leq k$, $f_i: J_1 \ra J_i$ be bijections and we form the set 
\ali{E := \set{\ga f_2(\ga)...f_k(\ga) : \ga \in J_1}.}
Then $|E| = |\Ga|$.  Since for all $\ga \in J_i$, $1 \leq i \leq k$, the order of $\ga$ is a power of $p_i$, if $\ga \in J_1$, $f_i(\ga) \in \Ga_{p_1^{m_1}}$ for all $2 \leq i \leq k$ and hence 
\ali{\pi_{p_1^{m_1}}(\ga f_2(\ga)...f_k(\ga)) = \pi_{p_1^{m_1}}(\ga).}
Similarly,
\ali{ \pi_{p_i^{m_i}}(\ga f_2(\ga)...f_k(\ga)) = \pi_{p_i^{m_i}}(f_i(\ga)),}
for all $2 \leq i \leq k$. Hence, $\pi_{p_i^{m_i}}(E) = \pi_{p_i^{m_i}}(J_i) = S_i$ for $1 \leq i \leq k$. By \cref{cor3.4}, $E$ is $N$-PR.
\end{proof}

\begin{rem}
In \cite{2012paper} the terminology ``$N$-large'' sets is introduced. A set $E \subset \Ga$ is $N$-large if $|Q_{N}(E)| < |E|$ where $Q_{N}: \Ga \ra \Ga / H_{N}$ is the quotient map and $H_{N} \subset \Ga$ is the subgroup of elements of orders divisible by $N$.

Theorem 2.2 in \cite{2012paper} states that if $E \subset \Ga$ and $N$ is the smallest integer for which $E$ is $N$-large, then for all primes powers $p^n$ dividing $N$ there exists a weak $|1-e^{\pi i/p^n}|$-Kronecker subset $F \subset E$ with $|F| = |E|$.

First, we note that the assumption in \cref{cor2.7} is weaker than the assumption in Theorem 2.2 in \cite{2012paper}; specifically, if $E$ is infinite and $N$-large for minimal $N = p^mp_1^{m_1}...p_k^{m_k}$ and $n \leq m$ for some integer $n$, then $|\pi_{p^n}(E)| = |E|$. To see this, we argue by contradiction and assume $|\pi_{p^n}(E)| < |E|$. Let $M := p^{n-1}p_1^{m_1}...p_k^{m_k} < N$. If $\ga \in \Ga_{p^n} \cap H_N$ and $k$ is the order of $\ga$, then $k$ divides $N$ while $p^n$ does not divide $k$. This implies $k$ divides $M$ and hence $\ga \in H_M$. Thus $\Ga_{p^n} \cap H_N = H_M$. As a result, the map 
\ali{ T: Q_M(E) &\ra Q_N(E) \times \pi_{p^n}(E) \\
       T(Q_M(\ga)) &= (Q_N(\ga), \pi_{p^n}(\ga)).}
is well-defined and injective. Thus, if $|\pi_{p^n}(E)| < |E|$, then 
\ali{ |Q_M(E)| \leq |Q_N(E)| |\pi_{p^n}(E)| < |E| |E| = |E|}
for infinite $E$. Thus $E$ is $M$-large and this contradicts the assumption that $N$ is minimal. 

Moreover, recall that $p^n$-PR sets are special weak $|1-e^{\pi i/p^n}|$-Kronecker sets. This shows \cref{cor2.7} improves Theorem 2.2 in \cite{2012paper} when $E$ is uncountable.

\end{rem}

\end{document}